\documentclass[leqno]{article}
\usepackage{graphicx,amsmath,amsfonts,amscd,amsthm,amssymb,eucal,psfrag,color,url,bm}

\newtheorem{thm}{Theorem}[section]
\newtheorem{cor}[thm]{Corollary}

\DeclareMathOperator{\lcm}{lcm}

\usepackage{tikz}
\usetikzlibrary{arrows,positioning} 
\tikzset{
    >=stealth',
    punkt/.style={
           rectangle,
           rounded corners,
           draw=black, very thick,
           text width=6.5em,
           minimum height=2em,
           text centered},
    pil/.style={
           ->,
           thick,
           shorten <=2pt,
           shorten >=2pt,}
}

\def\lf{\left\lfloor}   
\def\rf{\right\rfloor}

\begin{document}

\oddsidemargin 16.5mm
\evensidemargin 16.5mm
\thispagestyle{plain}

\vspace{1cc}

\baselineskip=17pt

\title{Partial Franel sums}

\author{R.~Tom\'as\\
  CERN\\
  CH 1211 Geneva 23, Switzerland\\
 E-mail: rogelio.tomas@cern.ch}

\date{\today}

\maketitle

\abstract{
 Analytical expressions are derived for the position of irreducible fractions
  in the Farey sequence $F_N$ of order $N$ for a particular choice
  of $N$. The asymptotic behaviour is derived obtaining a lower
  error bound than in previous results when these fractions are in
  the vicinity of $0/1$, $1/2$ or $1/1$.
  
  Franel's famous formulation of Riemann's hypothesis uses
  the summation of distances between irreducible fractions and evenly
  spaced points in $[0,1]$. A partial Franel sum is defined here
  as a summation of these distances over a subset of fractions in $F_N$.
  The partial Franel sum in the range $[0, i/N]$, with $N={\rm lcm}(1,2,...,i)$
  is shown  here to grow as $O(\log(N)\delta_B(\log N))$, where $\delta_B(x)$ is a decreasing function.
  Other partial Franel sums are also explored.
}

\section{Introduction and main results}
The Farey sequence $F_N$ of order $N$  is an ascending sequence of irreducible 
fractions between 0 and 1 whose denominators do not exceed $N$~\cite{hardy}.
Riemann's hypothesis implies that the irreducible fractions  tend to be regularly distributed in $[0,1]$. A formulation
of this statement follows~\cite{Franel},
\begin{equation}
\sum_{n=1}^{|F_N|}\left|F_{N}(n)-\frac{n}{|F_N|}\right|= O\left(N^{\frac{1}{2}+\epsilon}\right)\ ,
\label{Riemann} \nonumber
\end{equation}
where $F_N(n)$ is the $n^{\rm th}$ irreducible fraction in $F_N$.
Here we define the partial Franel sum in the range $[a_1/b_1, a_2/b_2]$ as
\begin{equation}
P\left(\frac{a_1}{b_1},\frac{a_2}{b_2}\right)=\sum_{n=I_N(a_1/b_1)}^{I_N(a_2/b_2)}\left|F_{N}(n)-\frac{n}{|F_N|}\right| \ ,
\label{partial_franel} \nonumber
\end{equation}
where  $\displaystyle I_N\left(a/b\right)$ is the position that $a/b$ occupies in $F_N$.
In~\cite{dress} the upper bound of the distance 
$\left|F_{N}(n)-{n}/{|F_N|}\right|$
 is established to be $1/N$
and to be located at $F_N(2)=1/N$. This motivates the study of partial
Franel sums in  ranges including $1/N$.
Furthermore, another equivalent formulation of the Riemann's hypothesis involving
sums over irreducible fractions in the range $[0,1/4]$ follows~\cite{kanemitsu},
\begin{equation}
\sum_{n=1}^{I_N(1/4)}\left(F_{N}(n)-\frac{I_N(1/4)}{2|F_N|}\right)= O\left(N^{\frac{1}{2}+\epsilon}\right)\ , \nonumber
%\label{Riemannkane} 
\end{equation}
showing again the relevance of the vicinity of $1/N$.

In~\cite{motif}, Chapter 6, it is attempted to find a closed expression for
the $i^{th}$ fraction in $F_N$ ending in an ``analytical hole''. This paper achieves this goal for fractions in the range $[0, i/N]$, with $N={\rm lcm}(1,2,...,i)$ as explained in the following. Note that  $N=\lcm(1,2,...,i)=e^{\psi(i)}$, 
where $\psi(i)$ is the second Chebyshev function that fulfills the property $\psi(i) = (1+o(1))i$, and hence
$ i =(1 + o(1)) \log N $.

Let the subsequence $F_{N}^{a_1/b_1,\, a_2/b_2}$  of $F_N$, contain all the fractions of $F_N$ in $[a_1/b_1,\ a_2/b_2]$.
The cardinality of $F_{N}^{a_1/b_1,\, a_2/b_2}$ is well known to be~\cite{cobeli}
\begin{equation}
\left|F_N^{a_1/b_1,\, a_2/b_2}\right|=\frac{3}{\pi^2}\left(\frac{a_2}{b_2}-\frac{a_1}{b_1}\right)N^2 + O(N \log N)\ . \nonumber
\end{equation}

As $\displaystyle I_N\left(a_2/b_2\right)$ is the position that $a_2/b_2$ occupies in  $F_N$, it follows that
\begin{equation}
I_N\left(\frac{a_2}{b_2}\right) = \left|F_{N}^{0/1,\, a_2/b_2}\right|=\frac{3}{\pi^2}\frac{c}{d}N^2 + O(N \log N)\ . \label{wellknown} 
\end{equation}
A first result of this paper is the derivation of an analytical expression
for  $I_N\left(1/q\right)$ where $N={\rm lcm(1,2,...,i)}$ and $N/i \leq q \leq N$
as
\begin{equation}
I_N\left(\frac{1}{q} \right) = 2+ N\sum_{j=1}^{i}\frac{\varphi(j)}{j} - q\Phi(i)\ ,
\nonumber
\end{equation}
where $\varphi(i)$ is the Totient function and $\Phi(i)$ is the summatory Totient function.
To reach this relation a series of bijections
between $F_{i'}$, with $i'\leq i$, and subsequences 
of $F_N$ are established covering all elements in  $F_N^{0/1,\, 1/q}$. 
Thanks to these bijections the cardinality of $F_N^{0/1,\, 1/q}$ can be expressed as function 
of all $|F_{i'}|$.
These bijections are illustrated in
Table~\ref{tab1} for $N=\lcm(1,2,...,5)=60$.
This result is used to derive the equivalent asymptotic estimate of~(\ref{wellknown}) with a smaller residual error:
\begin{eqnarray}
  I_N\left(\frac{1}{q}\right) &=& 
  \frac{3}{\pi^2}q\left(\frac{N^2}{q^2} -\left\{ \frac{N}{q}  \right\}^2   \right) +  O\left(N\delta_A\left(\lf\frac{N}{q}\rf\right)\right)\ ,
\nonumber
\end{eqnarray}
where $\left\{x\right\}=x - \lf x\rf$ and $\delta_A(x)$ is a decreasing function defined as
\begin{eqnarray}
\delta_A(x)&=&\exp\left(-A\frac{\log^{0.6}x}{(\log\log x)^{0.2}}\right)\ ,\label{delta}
\end{eqnarray}
where $A>0$. %in detail in Section~\ref{Res}.

Using this result the partial Franel sum in the range $[0,$ $1/(N/i)]$ is shown to be

\begin{equation}
P\left(\frac{0}{1},\frac{1}{N/i}\right)=O(\log (N)\delta_B(\log N))   \ , \nonumber
%\label{partialsumfulltheoEq} 
\end{equation}
with $0<B<A$ and again $N=\lcm(1,2,...,i)$.
This partial Franel sum, therefore, grows strictly slower than $O(\log N)$.
If we would assume the Riemann hypothesis and a uniform distribution density of Farey elements in $[0,1]$
 we would expect this partial Franel sum to actually decrease as $O(\log (N)/N^{1/2-\epsilon})$. An equivalent result is obtained for partial Franel sums in ranges including $1/2$. The generalization to compute partial Franel sums in the vicinity of any irreducible fraction is explored.
Earlier results of this work were applied to resonance diagrams~\cite{tomas,arxiv}.

\begin{table}
\renewcommand{\arraystretch}{1.15}\addtolength{\tabcolsep}{1.2pt}
\caption{Correspondance between elements in $F_{i'}$,
with $0<i'<5$ and first 90 elements in $F_{N}$, given by
$u/l=h/(hq-k)$ with $N/(i'+1)<q\leq N/i'$ and $N=\lcm(1,2,3,4,5)=60$. Note that the images  of elements $1/1$ and $0/1$ of adjacent maps are equal and only the $1/1$ case is shown on the table. The illustrated maps originate from the map $M$ in Theorem~\ref{maptheo} with $a_1/b_1=0/1$ and $a_2/b_2=1/0$.}\label{tab1}
\centering
\resizebox{1.1\textwidth}{!}{%
\begin{tabular}{|ccccc||ccccc||ccccc|}\hline
  $i'$ & $q$ & $h/k $     & $u/l$     & $I_N(\frac{u}{l})$  &  $i'$ & $q$       & $h/k $    & $u/l$     & $I_N(\frac{u}{l})$  &  $i'$ & $q$       & $h/k $    & $u/l$     & $I_N(\frac{u}{l})$  \\
      &     & $\in F_{i'}$  & $\in F_N$  &            &      &           &$\in F_{i'}$  & $\in F_N$ &      &      &           &$\in F_{i'}$  & $\in F_N$ &            \\  \hline
- & -  &  -             &$\frac{0}{1}$          & 1  & 1 & 32 & $\frac{ 1 }{ 1 }$ & $\frac{ 1 }{ 31 }$ & 31 &  3 & 18 & $\frac{ 1 }{ 3 }$ & $\frac{ 3 }{ 53 }$ & 61 \\\hline
1 & 60 & $\frac{ 0 }{1}$&$\frac{\bm1}{\bm6\bm0}$& 2  & 1 & 31 & $\frac{1}{1}$ &$\frac{\bm1}{\bm3\bm0}$ & 32 &  3 & 18 & $\frac{ 1 }{ 2 }$ & $\frac{ 2 }{ 35 }$ & 62 \\\hline
1 & 60 & $\frac{ 1 }{ 1 }$ & $\frac{ 1 }{ 59 }$ & 3  & 2 & 30 & $\frac{ 1 }{ 2 }$ & $\frac{ 2 }{ 59 }$ & 33 &  3 & 18 & $\frac{ 2 }{ 3 }$ & $\frac{ 3 }{ 52 }$ & 63 \\\hline
1 & 59 & $\frac{ 1 }{ 1 }$ & $\frac{ 1 }{ 58 }$ & 4  & 2 & 30 & $\frac{ 1 }{ 1 }$ & $\frac{ 1 }{ 29 }$ & 34 &  3 & 18 & $\frac{ 1 }{ 1 }$ & $\frac{ 1 }{ 17 }$ & 64 \\\hline
1 & 58 & $\frac{ 1 }{ 1 }$ & $\frac{ 1 }{ 57 }$ & 5  & 2 & 29 & $\frac{ 1 }{ 2 }$ & $\frac{ 2 }{ 57 }$ & 35 &  3 & 17 & $\frac{ 1 }{ 3 }$ & $\frac{ 3 }{ 50 }$ & 65 \\\hline
1 & 57 & $\frac{ 1 }{ 1 }$ & $\frac{ 1 }{ 56 }$ & 6  & 2 & 29 & $\frac{ 1 }{ 1 }$ & $\frac{ 1 }{ 28 }$ & 36 &  3 & 17 & $\frac{ 1 }{ 2 }$ & $\frac{ 2 }{ 33 }$ & 66 \\\hline
1 & 56 & $\frac{ 1 }{ 1 }$ & $\frac{ 1 }{ 55 }$ & 7  & 2 & 28 & $\frac{ 1 }{ 2 }$ & $\frac{ 2 }{ 55 }$ & 37 &  3 & 17 & $\frac{ 2 }{ 3 }$ & $\frac{ 3 }{ 49 }$ & 67 \\\hline
1 & 55 & $\frac{ 1 }{ 1 }$ & $\frac{ 1 }{ 54 }$ & 8  & 2 & 28 & $\frac{ 1 }{ 1 }$ & $\frac{ 1 }{ 27 }$ & 38 &  3 & 17 & $\frac{ 1 }{ 1 }$ & $\frac{ 1 }{ 16 }$ & 68 \\\hline
1 & 54 & $\frac{ 1 }{ 1 }$ & $\frac{ 1 }{ 53 }$ & 9  & 2 & 27 & $\frac{ 1 }{ 2 }$ & $\frac{ 2 }{ 53 }$ & 39 &  3 & 16 & $\frac{ 1 }{ 3 }$ & $\frac{ 3 }{ 47 }$ & 69 \\\hline
1 & 53 & $\frac{ 1 }{ 1 }$ & $\frac{ 1 }{ 52 }$ & 10 & 2 & 27 & $\frac{ 1 }{ 1 }$ & $\frac{ 1 }{ 26 }$ & 40 &  3 & 16 & $\frac{ 1 }{ 2 }$ & $\frac{ 2 }{ 31 }$ & 70 \\\hline
1 & 52 & $\frac{ 1 }{ 1 }$ & $\frac{ 1 }{ 51 }$ & 11 & 2 & 26 & $\frac{ 1 }{ 2 }$ & $\frac{ 2 }{ 51 }$ & 41 &  3 & 16 & $\frac{ 2 }{ 3 }$ & $\frac{ 3 }{ 46 }$ & 71 \\\hline
1 & 51 & $\frac{ 1 }{ 1 }$ & $\frac{ 1 }{ 50 }$ & 12 & 2 & 26 & $\frac{ 1 }{ 1 }$ & $\frac{ 1 }{ 25 }$ & 42 &  3 & 16 &$\frac{1}{1}$ & $\frac{\bm1}{\bm1\bm5}$ & 72 \\\hline
1 & 50 & $\frac{ 1 }{ 1 }$ & $\frac{ 1 }{ 49 }$ & 13 & 2 & 25 & $\frac{ 1 }{ 2 }$ & $\frac{ 2 }{ 49 }$ & 43 &  4 & 15 & $\frac{ 1 }{ 4 }$ & $\frac{ 4 }{ 59 }$ & 73 \\\hline
1 & 49 & $\frac{ 1 }{ 1 }$ & $\frac{ 1 }{ 48 }$ & 14 & 2 & 25 & $\frac{ 1 }{ 1 }$ & $\frac{ 1 }{ 24 }$ & 44 &  4 & 15 & $\frac{ 1 }{ 3 }$ & $\frac{ 3 }{ 44 }$ & 74 \\\hline
1 & 48 & $\frac{ 1 }{ 1 }$ & $\frac{ 1 }{ 47 }$ & 15 & 2 & 24 & $\frac{ 1 }{ 2 }$ & $\frac{ 2 }{ 47 }$ & 45 &  4 & 15 & $\frac{ 1 }{ 2 }$ & $\frac{ 2 }{ 29 }$ & 75 \\\hline
1 & 47 & $\frac{ 1 }{ 1 }$ & $\frac{ 1 }{ 46 }$ & 16 &  2 & 24 & $\frac{ 1 }{ 1 }$ & $\frac{ 1 }{ 23 }$ & 46&  4 & 15 & $\frac{ 2 }{ 3 }$ & $\frac{ 3 }{ 43 }$ & 76 \\\hline
1 & 46 & $\frac{ 1 }{ 1 }$ & $\frac{ 1 }{ 45 }$ & 17 &  2 & 23 & $\frac{ 1 }{ 2 }$ & $\frac{ 2 }{ 45 }$ & 47&  4 & 15 & $\frac{ 3 }{ 4 }$ & $\frac{ 4 }{ 57 }$ & 77 \\\hline
1 & 45 & $\frac{ 1 }{ 1 }$ & $\frac{ 1 }{ 44 }$ & 18 &  2 & 23 & $\frac{ 1 }{ 1 }$ & $\frac{ 1 }{ 22 }$ & 48&  4 & 15 & $\frac{ 1 }{ 1 }$&$\frac{1}{14}$& 78 \\\hline
1 & 44 & $\frac{ 1 }{ 1 }$ & $\frac{ 1 }{ 43 }$ & 19 &  2 & 22 & $\frac{ 1 }{ 2 }$ & $\frac{ 2 }{ 43 }$ & 49&  4 & 14 & $\frac{ 1 }{ 4 }$ & $\frac{ 4 }{ 55 }$ & 79 \\\hline
1 & 43 & $\frac{ 1 }{ 1 }$ & $\frac{ 1 }{ 42 }$ & 20 &  2 & 22 & $\frac{ 1 }{ 1 }$ & $\frac{ 1 }{ 21 }$ & 50&  4 & 14 & $\frac{ 1 }{ 3 }$ & $\frac{ 3 }{ 41 }$ & 80 \\\hline
1 & 42 & $\frac{ 1 }{ 1 }$ & $\frac{ 1 }{ 41 }$ & 21 &  2 & 21 & $\frac{ 1 }{ 2 }$ & $\frac{ 2 }{ 41 }$ & 51&  4 & 14 & $\frac{ 1 }{ 2 }$ & $\frac{ 2 }{ 27 }$ & 81 \\\hline
1 & 41 & $\frac{ 1 }{ 1 }$ & $\frac{ 1 }{ 40 }$ & 22 &  2 & 21 & $\frac{ 1}{1}$ &$\frac{\bm1}{\bm2\bm0}$& 52&  4 & 14 & $\frac{ 2 }{ 3 }$ & $\frac{ 3 }{ 40 }$ & 82 \\\hline
1 & 40 & $\frac{ 1 }{ 1 }$ & $\frac{ 1 }{ 39 }$ & 23 &  3 & 20 & $\frac{ 1 }{ 3 }$ & $\frac{ 3 }{ 59 }$ & 53&  4 & 14 & $\frac{ 3 }{ 4 }$ & $\frac{ 4 }{ 53 }$ & 83 \\\hline
1 & 39 & $\frac{ 1 }{ 1 }$ & $\frac{ 1 }{ 38 }$ & 24 &  3 & 20 & $\frac{ 1 }{ 2 }$ & $\frac{ 2 }{ 39 }$ & 54&  4 & 14 & $\frac{ 1 }{ 1 }$ & $\frac{ 1 }{ 13 }$ & 84 \\\hline
1 & 38 & $\frac{ 1 }{ 1 }$ & $\frac{ 1 }{ 37 }$ & 25 &  3 & 20 & $\frac{ 2 }{ 3 }$ & $\frac{ 3 }{ 58 }$ & 55&  4 & 13 & $\frac{ 1 }{ 4 }$ & $\frac{ 4 }{ 51 }$ & 85 \\\hline
1 & 37 & $\frac{ 1 }{ 1 }$ & $\frac{ 1 }{ 36 }$ & 26 &  3 & 20 & $\frac{ 1 }{ 1 }$ & $\frac{ 1 }{ 19 }$ & 56&  4 & 13 & $\frac{ 1 }{ 3 }$ & $\frac{ 3 }{ 38 }$ & 86 \\\hline
1 & 36 & $\frac{ 1 }{ 1 }$ & $\frac{ 1 }{ 35 }$ & 27 &  3 & 19 & $\frac{ 1 }{ 3 }$ & $\frac{ 3 }{ 56 }$ & 57&  4 & 13 & $\frac{ 1 }{ 2 }$ & $\frac{ 2 }{ 25 }$ & 87 \\\hline
1 & 35 & $\frac{ 1 }{ 1 }$ & $\frac{ 1 }{ 34 }$ & 28 &  3 & 19 & $\frac{ 1 }{ 2 }$ & $\frac{ 2 }{ 37 }$ & 58&  4 & 13 & $\frac{ 2 }{ 3 }$ & $\frac{ 3 }{ 37 }$ & 88 \\\hline
1 & 34 & $\frac{ 1 }{ 1 }$ & $\frac{ 1 }{ 33 }$ & 29 &  3 & 19 & $\frac{ 2 }{ 3 }$ & $\frac{ 3 }{ 55 }$ & 59&  4 & 13 & $\frac{ 3 }{ 4 }$ & $\frac{ 4 }{ 49 }$ & 89 \\\hline
1 & 33 & $\frac{ 1 }{ 1 }$ & $\frac{ 1 }{ 32 }$ & 30 &  3 & 19 & $\frac{ 1 }{ 1 }$ & $\frac{ 1 }{ 18 }$ & 60&  4 & 13 & $\frac{ 1}{1}$&$\frac{\bm1}{\bm1\bm2}$ & 90 \\\hline
\end{tabular} 
}
\end{table}

\section{Definitions}

We say that two elements of a Farey sequence, $a_1/b_1$ and $a_2/b_2$, form a Farey pair if $|a_1b_2-a_2b_1|=1$. 
In this report we exceptionally allow  $0/1$ and $1/0$
to form a Farey pair even if $1/0$ is not a proper fraction.
The mediant of a Farey pair, $a_1/b_1$ and $a_2/b_2$, is given by
\begin{equation}
    \frac{a_1 + a_2}{b_1+b_2}\nonumber
\end{equation}
which is an irreducible fraction existing between $a_1/b_1$ and $a_2/b_2$ and forms two Farey pairs
with $a_1/b_1$ and $a_2/b_2$.

\section{Results}\label{Res}

%\begin{cor}\label{gcdneigh}
%  Let $a_1/b_1$ and $a/b$ be Farey neighbours and
%  $u/l$ being a Farey fraction such
%  \[\frac{a_1}{b_1} < \frac{u}{l}  < \frac{a}{b}\ , \]
%  then
%  ${\rm gcd}(la - ub ,   b_1 u-la_1)=1$.
%\end{cor}

% \chi -> a_1
% \eta -> b_1
% a -> a_2
% b -> b_2

\begin{thm}\label{maptheo}
Let $a_1/b_1$ and $a_2/b_2$ be a Farey pair with $b_1>b_2$.
Let $N$ be multiple of $b_1 i(i+1)$ with $i$ being a natural number such $0<i<N$.
Let $q$ be an integer 
fulfilling  
\begin{equation}\nonumber
\displaystyle \frac{N}{b_1(i+1)} < q \leq \frac{N}{b_1 i}
\ \ \ \
and \ \ \ \ b_1q+b_2 \leq N\ .
\end{equation}
Let $F'_i$ be defined as
\begin{equation}
F'_i = \left\{ \frac{h}{k}: \frac{h}{k}\in F_i\, , \,\, k(b_1 q+b_2)-b_1 h \leq N  \right\}\ . \nonumber
\end{equation}
There is a bijective map $M$ between $F'_i$ and $\displaystyle F_{N}^{ \frac{a_1 q+a_2}{b_1 q + b_2} ,\, \frac{a_1 (q-1)+a_2}{b_1 (q-1)+b_2}}$,  given by
\begin{equation}
M:\ F'_i \rightarrow F_{N}^{ \frac{a_1 q+a_2}{b_1 q + b_2} ,\, \frac{a_1 (q-1)+a_2}{b_1 (q-1)+b_2}}\ , \ \ \ \ \ \ \frac{h}{k} \mapsto \frac{k(a_1 q+a_2)-a_1 h}{k(b_1 q+b_2)-b_1 h }\label{map}\ . \nonumber
\end{equation}
\begin{equation}
M^{-1}:\ F_{N}^{ \frac{a_1 q+a_2}{b_1 q + b_2} ,\, \frac{a_1 (q-1)+a_2}{b_1 (q-1)+b_2}} \rightarrow F'_i\ , \ \ \ \ \ \ \frac{u}{l} \mapsto  \frac{q(b_1 u- la_1)+b_2u-la_2}{b_1 u- la_1}\label{mapinv}\ . \nonumber
\end{equation}
The bijective map is order-preserving when $a_2/b_2 > a_1/b_1$
and order-inverting when $a_2/b_2 < a_1/b_1$.
\end{thm}\noindent

\begin{proof}
%{\color{red}  For the particular case of $a_1/b_1=0/1$, $a_2/b_2$ can be taken as $1/0$
%  since indeed they are neighbours in the extended Farey diagram. For other cases $a/b$ may be computed with existing formulae~\cite{Matveev}. 
%  }
 We first demonstrate that $M$ is injective.  
  $\frac{a_1 q+a_2}{b_1 q + b_2}$ and  $\frac{a_1 (q-1)+a_2}{b_1 (q-1)+b_2}$
  form a Farey pair since  $a_1/b_1$ and $a_2/b_2$ form a Farey pair:
  \[|(a_1 q +a_2)(b_1(q-1)+b_2) - (b_1 q+b_2)(a_1(q-1)+a_2)| = |b_2 a_1 -a_2 b_1| =1 \ .\]
Let $u/l$ be the image of $h/k$ under $M$,
\begin{equation}
\frac{u}{l} = \frac{k(a_1 q+a_2)-a_1 h}{k(b_1 q+b_2)-b_1 h }    \ .\nonumber
\end{equation}
By virtue of this expressison $u/l$ is  obtained by applying the mediant operation successively between $\frac{a_1 q+a_2}{b_1 q + b_2}$ and  $\frac{a_1 (q-1)+a_2}{b_1 (q-1)+b_2}$ 
in the same fashion as  $h/k$ is obtained by applying the mediant between $0/1$ and $1/1$, meaning
\begin{eqnarray}
 \frac{h}{k}&=&\frac{(k-h)\cdot0 + h \cdot 1}{(k-h)\cdot 1 + h\cdot 1 }    \ , \nonumber\\
 \frac{u}{l} &=& \frac{(k-h)\cdot(a_1 q+a_2) + h\cdot(a_1 (q-1)+a_2)}{(k-h)\cdot(b_1 q + b_2)+ h\cdot(b_1 (q-1)+b_2)} \ . \nonumber
\end{eqnarray}
Therefore $u/l$ is a Farey fraction in the interval of interest:
\begin{equation}
\left[ \frac{a_1 q+a_2}{b_1 q + b_2} ,\, \frac{a_1 (q-1)+a_2}{b_1 (q-1)+b_2}\right]\ . \nonumber
\end{equation}
$u/l$ belongs to $F_N$  by definition of the domain $F'_i$, meaning that $h/k$ belongs to $F'_i$ if $l\leq N$.

Now we demonstrate that $M^{-1}$ is also injective. Let $u/l$ belong to 
$F_{N}^{ \frac{a_1 q+a_2}{b_1 q + b_2} ,\, \frac{a_1 (q-1)+a_2}{b_1 (q-1)+b_2}}$
%and $h/k$ be the image of $u/l$ under $M^{-1}$.
and assume $a_2/b_2>a_1/b_1$, so that
\begin{equation}
\frac{a_1 q+a_2}{b_1 q + b_2} \leq %\frac{k(a_1 q+a_2)-a_1 h}{k(b_1 q+b_2)-b_1 h }   = 
\frac{u}{l}   \leq \frac{a_1 (q-1)+a_2}{b_1 (q-1)+b_2}\  \label{comparul} \ .
    \end{equation}
Let $h/k$ be the image of $u/l$ under $M^{-1}$,
\begin{equation}
   \frac{h}{k} = \frac{q(b_1 u- la_1)+ub_2-la_2}{b_1 u- la_1}\ . \label{invmap}
\end{equation}
This equality implies gcd$(h,k)$= gcd$(ub_2-la_2,b_1 u - la_1)$. 
Since 
\begin{equation}
\gcd(u,l)=\gcd(a_1,b_1)=\gcd(a_2,b_2)=1 \nonumber
\end{equation}
and $a_2b_1-a_1b_2=1$
then gcd$(h,k)=1$ according to the property in~\cite{nntdm} and, hence, $h/k$ is an irreducible fraction. Furthermore, operating with the inequalities in~(\ref{comparul}):
\begin{equation}
q(b_1u-la_1) \geq - (ub_2 -la_2 ) \geq (b_1u -l a_1)(q-1) \nonumber
\end{equation}
and therefore $0\leq h \leq k$.

From relations~(\ref{comparul}) and~(\ref{invmap})
\begin{equation}
  k= b_1 u -la_1 \leq b_1 l \frac{a_1 (q-1)+a_2}{b_1 (q-1)+b_2} - l a_1=  \frac{l}{b_1 (q-1)+b_2}\nonumber
\end{equation}
and using that $l\leq N$ and $q > \frac{N}{b_1(i+1)}$, hence $b_1(q-1) \geq \frac{N}{i+1}$,
\begin{equation}
k \leq \frac{N}{\frac{N}{i+1}+b_2} = \frac{i+1}{1 +\frac{i+1}{N}b_2} < i+1 \nonumber \ .
  \end{equation}
If $b_2>0$ this implies $k \leq i$ and gathering 
the above results $0\leq h \leq k \leq i$ and gcd$(h,k)$=1, hence $h/k \in F_i$. To demonstrate that $h/k$ belongs to $F'_i$ it is easy to verify that $k(b_1 q +b_2) -b_1h\leq N$.

If $b_2=0$ we are in the exceptional case included 
in this report of $a_1/b_1=0/1$ and $a_2/b_2=1/0$, that implies $h/k=(qu-l)/u$, note that $k=u$.
We only need to show that $k\leq i$ also in this case.
From the  inequalities in~(\ref{comparul}) and $\frac{N}{i} \geq q > \frac{N}{i+1}$,
\begin{equation}
\frac{i}{N} \leq \frac{1}{q}  \leq  \frac{u}{l} \leq \frac{1}{q-1} \leq %\frac{1}{\frac{N}{i+1}+1-1} =
\frac{i+1}{N} \ .
    \nonumber
\end{equation}
 $(i+1)/N$ is not an irreducible fraction,
 as $N$ is taken as a multiple of $i(i+1)$, and therefore it does not belong to $F_N$. Similarly for $i/N$ when $i>1$. In the range $[i/N, (i+1)/N]$ there cannot be  fractions with denominator $N$ other than $1/N$ when $i=1$. Therefore if $i=1$ we directly have $k=u\leq i$ and for $i>1$ we have that $l\leq N-1$ and hence 
 \begin{equation}
    k=u\leq l\frac{i+1}{N} \leq i \nonumber \ .
 \end{equation}
\end{proof}

\begin{cor}\label{cardinals}
  The cardinalities of $F_i$, $F'_i$ and $\displaystyle F_{N}^{ \frac{a_1 q+a_2}{b_1 q + b_2} ,\, \frac{a_1 (q-1)+a_2}{b_1 (q-1)+b_2}}$
  are related as follows:
  \begin{itemize}
  \item If $q=N/(b_1 i)$ then
    \begin{equation}
  |F_i| \geq |F'_i| = \left| F_{N}^{ \frac{a_1 q+a_2}{b_1 q + b_2} ,\, \frac{a_1 (q-1)+a_2}{b_1 (q-1)+b_2}}\right| > |F_i| - i \nonumber
\end{equation}
  \item If $q < N/(b_1 i)$ or $b_2=0$ then
    \begin{equation}
     |F_i| = |F'_i| = \left| F_{N}^{ \frac{a_1 q+a_2}{b_1 q + b_2} ,\, \frac{a_1 (q-1)+a_2}{b_1 (q-1)+b_2}}\right| \nonumber 
\end{equation}
  \end{itemize}
\end{cor}
\begin{proof}
  The first inequality is evident from the definition of $F'_i$. The first equality
  derives from the the bijective map
  in Theorem~\ref{maptheo}.

If $q=N/(b_1 i)$, let $u/l$ be the image of $h/k$ via the map $M$ in Theorem~\ref{maptheo},  then $l=k(N/i+b_2)-b_1 h$. 
To prove that $|F_i'|>|F_i|-i$ we should
count how many $h/k \in F_i$  
fulfill $k(N/i+b_2)-b_1 h > N$. Dividing  both sides
of the later inequality by $k$ and operating we obtain
\begin{equation}
b_2-b_1 \frac{h}{k} > \frac{N}{k} - \frac{N}{i} = N\frac{i-k}{ki}\ , \nonumber
\end{equation}
\begin{equation}
b_2 \geq b_2-b_1\frac{h}{k} > N\frac{i-k}{ki} 
%\geq b_1(i+1)\frac{i-k}{k} 
\geq 0  \ . \label{esta}
\end{equation}
To fulfill these inequalities it is required that $k=i$, otherwise for any $k<i$
and recalling that $N$ is a multiple of   $b_1 i (i+1)$:
\begin{equation}
N\frac{i-k}{ki} \geq b_1\frac{i+1}{k}(i-k)>b_1\ , \nonumber
\end{equation}
and inequalities in~(\ref{esta}) cannot be fulfilled as $b_2 < b_1$ (from assumption in Theorem~\ref{maptheo}).
Then  $k=i$ implies $h/i<b_2/b_1<1$ and in $F_i$ there are fewer than $i$ irreducible fractions of the form $h/i$ below $b_2/b_1$, hence  $|F_i'|>|F_i|-i$.

If $q<N/(b_1 i)$ we define $g>0$ such that $q=N/(b_1 i)-g$, then $l=k(N/i-gb_1+b_2)-b_1 h $ and we need to count how many $h/k$ in $F_i$ have $l>N$,
\begin{equation}
\ 
b_2-b_1\frac{h}{k}- b_1 g  > N\frac{i-k}{ki}% \geq 0 
\ ,\nonumber
\end{equation}
and there are no $h/k$ which can fulfill this equation as $b_2-b_1g < 0$, hence $|F_i|=|F_i'|$ when $q<N/(b_1 i)$.

If $b_2=0$ we should show that there are no $h/k$ in $|F_i|$ fulfilling $kb_1 q-b_1 h > N$.
The largest possible value of $q$ is $N/(b_1 i)$ and therefore
$kb_1 q-b_1 h \leq kN/i-b_1 h < N$, for $i>1$, so there is no $h/k$
fulfilling the previous condition and $|F_i|=|F'_i|$.
Note that $i=1$ and $h/k=0/1$ would not have given $kb_1 q-b_1 h > N$ as $b_1q+b_2 \leq N$ from the assumptions in Theorem~\ref{maptheo}.
%This case with  $b_2=0$ is relevant when $a_1/b_1=0/1$ and $a_2/b_2=1/0$.
\end{proof}
%
%
%
% a_1 -> a_1
% b_1 -> b_1
% a -> a_2
% b -> b_2
\begin{thm} \label{Iq} Let $N=b_1 {\rm lcm}(1,2,...,i_{max})$, $\displaystyle \frac{N}{b_1(i+1)} < q \leq \frac{N}{b_1 i}$, with   $a_1/b_1$ and $a_2/b_2$ forming a Farey pair, $b_1>b_2$ and   $i < i_{max}$  then:
  \begin{itemize}
\item  For $b_1 > 1$:
\begin{equation}
I_N\left(\frac{a_1 q+a_2}{b_1 q+b_2} \right) = I_N\left(\frac{a_1}{b_1}\right) +s \left(\frac{N}{b_1}\sum_{j=1}^{i}\frac{\varphi(j)}{j} - q\Phi(i)\right) + O(i^2)\ ,
%\nonumber
\label{Itheo}
\end{equation}
with $s=+1$ when $a_1/b_1 < a_2/b_2$ and $s=-1$ otherwise. 
\item For $a_1/b_1 = 0/1$ and $a_2/b_2=1/0$:
\begin{equation}
I_N\left(\frac{1}{q} \right) = 2+ N\sum_{j=1}^{i}\frac{\varphi(j)}{j} - q\Phi(i)\ .
\nonumber
\end{equation}
\end{itemize}
 \end{thm}\noindent
\begin{proof}
  To simplify equations we assume $s=+1$ in the following.
  We count the number of elements in $F_N^{\frac{a_1}{b_1},\frac{a_1 q+a_2}{b_1 q+b_2}}$
  using the bijective maps described in Theorem~\ref{maptheo}
  and adding up the cardinalities of the sets involved from Corolary~\ref{cardinals}.
  Thanks to the fact that $N$ is multiple 
  of all natural numbers $i'$ such that $i'\leq i$
  we can establish bijections between $F_i'$ and
  $F_N^{\frac{a_1p+a_2}{b_1p+b_2},\frac{a_1 (p-1)+a_2}{b_1 (p-1)+b_2}}$ where $p$ can take all values fulfilling 
  $\displaystyle \frac{N}{b_1(i'+1)} < p \leq \frac{N}{b_1 i'}$, covering all elements in $F_N^{\frac{a_1}{b_1},\frac{a_1 q+a_2}{b_1 q+b_2}}$
  when scannig over all $i'\leq i$ and the corresponding $p$. For a given $i'$
  the number of values $p$ takes is given by
  \begin{eqnarray}
     \frac{N}{b_1 i'} - \frac{N}{b_1(i'+1)} = \frac{N}{b_1}\left(\frac{1}{ i'} -\frac{1}{i'+1}\right) \nonumber \ .
    \end{eqnarray}
  In a first step we compute the number of elements in   $F_N^{\frac{a_1}{b_1},\frac{a_1 q'+a_2}{b_1 q'+b_2}}$ with $q'=N/(b_1 i)$,
  \begin{eqnarray}
    I_N\left(\frac{a_1 q'+a_2}{b_1 q'+b_2}\right) - I_N\left(\frac{a_1}{b_1}\right)   &=&\frac{N}{b_1} \sum_{i'=1}^{i-1} \left(\frac{1}{ i'} - \frac{1}{i'+1}\right)  (|F'_{i'}|-1) \nonumber\\
    &=& \frac{N}{b_1}\sum_{i'=1}^{i-1} \left[\left(\frac{1}{ i'}-\frac{1}{ i'+1}\right)\Phi(i') + O(i')\right]\nonumber\\
    &=&  \frac{N}{b_1} \sum_{j=1}^{i-1} \frac{\varphi(j)}{j}- \frac{N}{b_1} \frac{\Phi(i-1)}{i} + O(i^2)\ . \nonumber
  \end{eqnarray}
   In particular, when $b_2=0$ the term $O(i^2)$  does not appear according to Corolary~\ref{cardinals}.
   In a second step we compute the number of elements 
   in $F_N^{\frac{a_1 q'+a_2}{b_1 q'+b_2},\ \frac{a_1 q+a_2}{b_1 q+b_2} }$, that is  $\Phi(i)(q'-q)+O(i)$. Adding both contributions gives
\begin{eqnarray}
I_N\left(\frac{a_1 q+a_2}{b_1 q+b_2}\right) - I_N\left(\frac{a_1}{b_1}\right)  &=&    \frac{N}{b_1} \sum_{j=1}^{i-1} \frac{\varphi(j)}{j}- \frac{N}{b_1} \frac{\Phi(i-1)}{i} \nonumber\\ & & \ \ \ \ 
+ \Phi(i)\left(\frac{N}{b_1 i} - q \right)+  O(i^2) \nonumber\\
  &=&   \frac{N}{b_1} \sum_{j=1}^{i} \frac{\varphi(j)}{j} -q\Phi(i)+  O(i^2)\ ,\nonumber
\end{eqnarray}
which demonstrates the theorem for $s=1$.
For $s=-1$ following the same steps leads to the desired result. 
    \end{proof}

\begin{cor} \label{theoIeta}Let $N=b_1 {\rm lcm}(1,2,...,i_{max})$ and $\displaystyle \frac{N}{b_1(i+1)} < q \leq \frac{N}{b_1 i}$, with $i < i_{max}$    then
\begin{eqnarray}
I_N\left(\frac{a_1 q+a_2}{b_1 q+b_2}\right) &=&  I_N\left(\frac{a_1}{b_1}\right)  +s\frac{3}{\pi^2}q\left(\frac{N^2}{b_1^2q^2} -\left\{ \frac{N}{b_1q}  \right\}^2   \right) +  O(N\delta_A(i))\ ,
\nonumber%\label{ItheoO}
\end{eqnarray}
with $\delta_A(x)$ defined in~(\ref{delta}).
In particular for $a_1/b_1=0/1$ and $a_2/b_2=1/0$,
\begin{eqnarray}
  I_N\left(\frac{1}{q}\right) &=&  \frac{3}{\pi^2}q\left(\frac{N^2}{q^2} -\left\{ \frac{N}{q}  \right\}^2   \right) +  O(N\delta_A(i)) ,
\nonumber%\label{0/1}
\end{eqnarray}
and for $a_1/b_1=1/2$ and $a_2/b_2=1/1$,
\begin{eqnarray}
I_N\left(\frac{q+1}{2q+1}\right) &=& \frac{|F_N|}{2}  +\frac{3}{\pi^2}q\left(\frac{N^2}{2^2q^2} -\left\{ \frac{N}{2q}  \right\}^2   \right) +  O(N\delta_A(i)) \ .
\nonumber%\label{1/2}
\end{eqnarray}
\end{cor}

\begin{proof}
 The following known relations~\cite{kanemitsu2000,walfisz} are needed:
\begin{eqnarray}
\sum_{k=1}^{N}\varphi(k)&=&\frac{3}{\pi^2}N^2 +  E(N) , \label{asym}\\
\sum_{k=1}^{N} \frac{\varphi(k)}{k} &=&\frac{6}{\pi^2}N + H(N)\ ,\nonumber   \\
E(x) &=&  O\left(x\log^{2/3}x(\log\log x)^{4/3} \right)\ ,\nonumber\\
E(x)&=&xH(x)+O(x\delta_A(x))\ , \nonumber%\\
%\delta(x)&=&\exp\left(-A\frac{\log^{0.6}x}{(\log\log x)^{0.2}}\right)\ ,\nonumber
\end{eqnarray}
with $A>0$ and $\delta_A(x)$ is a decreasing factor.
From the definition of $i$, $q$ and $N$ it follows that
\begin{eqnarray}
  i &=& \lf\frac{N}{qb_1}\rf = \frac{N}{qb_1} + O(1)\ ,\nonumber\\
  i &<& i_{max} = (1+o(1))\log N/b_1\ . \nonumber
\end{eqnarray}
Inserting the above equalities in expression~(\ref{Itheo}) of  Theorem~\ref{Iq}, 
\begin{eqnarray}
I_N\left(\frac{a_1 q+a_2}{b_1 q+b_2}\right) &=& I_N\left(\frac{a_1}{b_1}\right) + s \frac{N}{b_1}\frac{6}{\pi^2}i -sq\frac{3}{\pi^2}i^2 + s\frac{N}{b_1}H(i) -sqE(i) + O(i^2)\nonumber \\
  &=& I_N\left(\frac{a_1}{b_1}\right) + s \frac{N}{b_1}\frac{6}{\pi^2}i -sq\frac{3}{\pi^2}i^2 + 
  sq(iH(i)-E(i))
  + O(i^2)\nonumber \\
   &=& I_N\left(\frac{a_1}{b_1}\right) + s \frac{N}{b_1}\frac{6}{\pi^2}i -sq\frac{3}{\pi^2}i^2 + O(N\delta_A(i))\nonumber \\
% &=& I_N\left(\frac{a_1}{b_1}\right)  +s \frac{3}{\pi^2}\frac{N^2}{b_1^2q} + O(N)  + sq(iH(i)-E(i)) \nonumber \\
 & =& I_N\left(\frac{a_1}{b_1}\right)  +s \frac{3}{\pi^2}\frac{N^2}{b_1^2q} - 
 s\frac{3}{\pi^2}q\left\{ \frac{N}{b_1q}  \right\}^2 +
 O(N\delta_A(i)) \ ,\nonumber
\end{eqnarray}
%where $\{x\}$ is the fractional part of $x$, defined as $\{x\}=x -\lf x \rf$.
%The last step is clarified with the following derivation:
%\begin{eqnarray}\nonumber
%\frac{N}{b_1}2i - qi^2 %=
%\frac{N}{b_1}2i - q i^2 + \frac{N^2}{b^2_1q} - \frac{N^2}{b_1q}
%=-q\left(\frac{N}{b_1q} - i \right)^2   + \frac{N^2}{b^2_1q} = \frac{N^2}{b^2_1q} + O(N\delta(i)),
%\end{eqnarray}
%{\color{red}
%\begin{eqnarray}\nonumber
%q\left\{\frac{N}{b_1q}  \right\}^2   = 
%q\left(\frac{N}{b_1q} - \lf\frac{N}{b_1q}\rf \right)^2 \leq \min\left(q ,\  q\left( \frac{N}{b_1 q} - 1\right)^2\right) 
%\leq \min\left(q ,\   \frac{N^2}{b_1^2 q} \right) 
%\end{eqnarray}
%}
%where we have used that $N\delta(i)< N$.
  \end{proof}

% i= int[ N/(q eta) ]
% 6/pi^2 N/eta i +H(i)*N/eta - qE(i) - q 3/pi^2 i^2 + O(i^2)
% 6/pi^2 N/eta i - 3/pi^2 N/eta i +O(qi) + O(i^2) +q(O(i delta(i)) +O(H(i)))

\begin{thm}\label{Ihktheo}
  Let $N$ be  $N = b_1 {\rm lcm}(1,2,...,i)$  then the partial Franel sum over all Farey fractions in the range
  $\displaystyle \left[\frac{a_1}{b_1} , \frac{a_1 \frac{N}{b_1 i} +a_2}{b_1 \frac{N}{b_1 i} +b_2} \right]$ gives:
  \begin{itemize}
\item For  $a_1/b_1=0/1$, $a_2/b_2=1/0$ and for $a_1/b_1=1/2$, $a_2/b_2=0/1$ :
\begin{equation}
P\left(\frac{0}{1},\frac{1}{N/i}\right)=\sum_{j=1}^{I_N\left( \frac{1}{N/i}\right)}\left|F_N(j)-\frac{j}{|F_N|}\right|=O(\log (N)\delta_B(\log N))   \ , \nonumber
%\label{partialsumfulltheoEq} 
\end{equation}
\begin{equation}
P\left(\frac{1}{2}, \frac{N/(2i)}{N/i+1} \right)=\sum_{j=I_N\left(\frac{1}{2}  \right)}^{I_N\left( \frac{N/(2i)}{N/i+1}\right)}\left|F_N(j)-\frac{j}{|F_N|}\right|=O(\log (N)\delta_B(\log N))   \ , \nonumber
%\label{partialsumfulltheoEq} 
\end{equation}
with $0<B<A$. The same result holds for $a_1/b_1=1/2$, $a_2/b_2=1/1$.

\item For $ b_1 > 2$ and $b_2 < b_1$:
\begin{equation}
  \sum_{j=I_N\left(\frac{a_1}{b_1}  \right)}^{I_N\left( \frac{a_1 \frac{N}{b_1 i} +a_2}{b_1 \frac{N}{b_1 i} +b_2}\right)}\left|F_N(j)-\frac{j}{|F_N|}\right|
   \leq
 \left|\frac{a_1}{b_1}  -   \frac{I_N\left(\frac{a_1}{b_1}\right)
   }{|F_N|} \right| O(iN)  + O(i\delta_B(i))
  \nonumber\ ,
\end{equation}
which cannot be further developed as no general expression for $I_N(a_1/b_1)$ is known.
\end{itemize}
\end{thm}

\begin{proof}
By virtue of Theorem~\ref{maptheo}
the partial Franel sum under study is written as
\begin{eqnarray}
 P\left(\frac{a_1}{b_1} , \frac{a_1 \frac{N}{b_1 i} +a_2}{b_1 \frac{N}{b_1 i} +b_2}  \right) = \sum_{i'=1}^{i-1}\sum_{q=\frac{N}{b_1(i'+1)}+1}^{\frac{N}{b_1 i'}}\sum_{n=2}^{|F'_{i'}|}
  \left|\frac{k(a_1 q+a_2)-a_1 h}{k(b_1 q+b_2)-b_1 h }  -  \frac{I_N\left( \frac{k(a_1 q+a_2)-a_1 h}{k(b_1 q+b_2)-b_1 h } \right) }{|F_N|}\right| \nonumber 
  \end{eqnarray}
  where the sum over $n$ runs over the elements $h/k$  in $F'_{i'}$, approximately $n=I_{i'}(h/k) + O(i')$.
 By virtue of Theorem~\ref{maptheo} and Corolary~\ref{theoIeta}
\begin{eqnarray}
 && I_N\left( \frac{k(a_1 q+a_2)-a_1 h}{k(b_1 q+b_2)-b_1 h } \right) =  I_N\left( \frac{a_1 q+a_2}{b_1 q+b_2}\right)
  +sI_{i'}\left(\frac{h}{k} \right) + O(i') \nonumber\\
%  \end{eqnarray}
%  \begin{eqnarray}
%  &=& I_N\left(\frac{a_1}{b_1}\right)  +s\frac{3}{\pi^2}\frac{N^2}{b_1^2q}  -    s\frac{3}{\pi^2}q\left\{\frac{N}{b_1q}\right\}^2 + sI_{i'}\left(\frac{h}{k} \right) + O(N\delta_A(i')) \nonumber\\
 &&= I_N\left(\frac{a_1}{b_1}\right)  +s\frac{3}{\pi^2}\frac{N^2}{b_1^2q}  -    s\frac{3}{\pi^2}q\left\{\frac{N}{b_1q}\right\}^2 + s\frac{3}{\pi^2}\lf\frac{N}{b_1q}\rf^2\frac{h}{k} + O(N\delta_A(i'))
\nonumber
\end{eqnarray}
where we have used  $i'=\lf\frac{N}{qb_1}\rf$% and $N/(b_1 i) < q < N/b_1$
.  Furthermore
\begin{eqnarray}
\frac{  I_N\left( \frac{k(a_1 q+a_2)-a_1 h}{k(b_1 q+b_2)-b_1 h } \right)}{|F_N|} &=& \frac{I_N\left(\frac{a_1}{b_1}\right)}{|F_N|}  +\frac{s}{b_1^2q}  -    \frac{sq}{N^2}\left\{\frac{N}{b_1q}\right\}^2 + \frac{s}{N^2}\lf\frac{N}{b_1q}\rf^2\frac{h}{k} + O\left(\frac{\delta_A(i')}{N}\right)
\nonumber
\end{eqnarray}
   
%  \begin{eqnarray}
% & & \sum_{i'=1}^{i-1}\sum_{q=\frac{N}{b_1(i'+1)}+1}^{\frac{N}{b_1 i'}}\sum_{j=2}^{|F'_{i'}|}
%  \left|\frac{k(a_1 q+a_2)-a_1 h}{k(b_1 q+b_2)-b_1 h }  - \frac{s}{b_1^2q} -   \frac{I_N\left(\frac{a_1}{b_1}\right)
%     + O(N) }{\frac{3}{\pi^2}N^2}\right| \nonumber \ .
%\end{eqnarray}
The Farey element inside the partial Franel sum is approximated as 
\begin{eqnarray}
\frac{k(a_1 q+a_2)-a_1 h}{k(b_1 q+b_2)-b_1 h } &=&
\frac{k(a_1 q+a_2)-a_1 h}{k(b_1 q+b_2)-b_1 h } -\frac{a_1}{b_1}+ \frac{a_1}{b_1} \nonumber \\ 
 &=& \frac{s}{b_1^2 q}\frac{1}{ 1+\frac{b_2}{qb_1}- \frac{h}{qk} } + \frac{a_1}{b_1} \nonumber\\
 &=& \frac{s}{b_1^2 q}\left(1 -\frac{b_2}{qb_1}+ \frac{h}{qk} \right)  + \frac{a_1}{b_1} + O(1/q^3) \nonumber%\\
% &=& \frac{s}{b_1^2 q}  + \frac{a_1}{b_1} + O(1/q^2)
% \nonumber 
\ ,
\end{eqnarray}
where we have used $(b_1 a_2 -a_1 b_2) =s$. The partial Franel sum under study becomes
  \begin{eqnarray}
& &\sum_{i'=1}^{i-1}\sum_{q=\frac{N}{b_1(i'+1)}+1}^{\frac{N}{b_1 i'}} |F'_{i'}| %\sum_{j=2}^{|F'_{i'}|}
  \left|\frac{a_1}{b_1}  -   \frac{I_N\left(\frac{a_1}{b_1}\right)
   }{|F_N|} - \frac{sb_2}{b_1^3q^2}  + 
   \frac{sq}{N^2}\left\{\frac{N}{b_1q}\right\}^2 +  
   O\left(\frac{\delta_A(i')}{N}\right)\right|\ , \nonumber
   \end{eqnarray}
   where the terms proportional to $1/q$ and $h/k$ have canceled out leaving a negligible residue. The sum over $n$ has been evaluated just by multiplying by $|F'_{i'}|$ as the dependency on $h/k$ disappeared.
   Evaluating the asymptotes of the sums of the individual terms within the   absolute value gives:
   
   \begin{eqnarray}
   \sum_{i'=1}^{i-1}\sum_{q=\frac{N}{b_1(i'+1)}+1}^{\frac{N}{b_1 i'}} |F'_{i'}| &=&
   O\left(iN\right)\ ,
   \nonumber\\
   \sum_{i'=1}^{i-1}\sum_{q=\frac{N}{b_1(i'+1)}+1}^{\frac{N}{b_1 i'}} \frac{|F'_{i'}|}{q^2} &=&
   O\left(\frac{i^4}{N}\right)\ ,
   \nonumber\\
   \sum_{i'=1}^{i-1}\sum_{q=\frac{N}{b_1(i'+1)}+1}^{\frac{N}{b_1 i'}} |F'_{i'}|\frac{q}{N^2} &=& O(\log i)\ ,  \nonumber\\
   \sum_{i'=1}^{i-1}\sum_{q=\frac{N}{b_1(i'+1)}+1}^{\frac{N}{b_1 i'}} |F'_{i'}|\frac{\delta_A(i')}{N} &=&  O\left(i\delta_B(i)\right) \nonumber\ ,
   \end{eqnarray}
with $0<B<A$. Keeping the two dominant terms gives
   \begin{eqnarray}
 P\left(\frac{a_1}{b_1} , \frac{a_1 \frac{N}{b_1 i} +a_2}{b_1 \frac{N}{b_1 i} +b_2}  \right) \leq
 \left|\frac{a_1}{b_1}  -   \frac{I_N\left(\frac{a_1}{b_1}\right)
   }{|F_N|} \right| O(iN)  + O(i\delta_B(i))\nonumber \ ,
    \end{eqnarray}
which is the searched result for $b_1>2$. For $1\leq b_1\leq 2$
\begin{equation}
  \frac{a_1}{b_1}  -   \frac{I_N\left(\frac{a_1}{b_1}\right)}{|F_N|} =0\nonumber
\end{equation}
and the theorem is demonstrated.
\end{proof}

%if we make $N=lcm^2([2,..,i])$, $N=e^{2i}$, $i=\log(N)/2$
%and assuming the $O(Nlog(N))$ for the approx
%$I_N(a_1/b_1)\approx a_1/b_1 N^2\frac{3}{pi^2}   + O(Nlog(N))$
%\begin{eqnarray}   
%   =O\left(\frac{1}{b_1}\log^2(N)\right) 
%\end{eqnarray}
%summing for all $b_1$:
%\begin{eqnarray}   
%  \sum_{b_1=1}^{i}O\left(\frac{\varphi(b_1)}{b_1}\log^2(N)\right)
%  =O(\log^3(N))
%\end{eqnarray}

%##################################
\end{document}